\documentclass[a4paper,11pt]{article}
\usepackage[utf8]{inputenc}
\usepackage[english]{babel}
\usepackage[ruled,section]{algorithm}
\usepackage{makeidx}
\usepackage[font=small]{caption}
\usepackage{amsmath}
\usepackage{amsthm}
\usepackage{amsfonts}
\usepackage{amssymb}
\usepackage{mathrsfs}
\usepackage{graphicx}
\usepackage[margin=20mm]{geometry}
\usepackage{enumerate}
\usepackage{indentfirst}
\usepackage{color}
\usepackage{authblk}
\usepackage{physics}
\usepackage{amsmath}
\usepackage{tikz}
\usepackage{mathdots}
\usepackage{yhmath}
\usepackage{cancel}
\usepackage{color}
\usepackage{siunitx}
\usepackage{array}
\usepackage{multirow}
\usepackage{amssymb}
\usepackage{gensymb}
\usepackage{tabularx}
\usepackage{extarrows}
\usepackage{booktabs}
\usepackage{subfig}
\usetikzlibrary{fadings}
\usetikzlibrary{patterns}
\usetikzlibrary{shapes}
\makeindex
\newcommand{\R}{\mathbb{R}}
\newcommand{\N}{\mathbb{N}}
\newcommand{\M}{\mathcal{M}}

\newcommand{\I}{I}
\newcommand{\A}{\mathcal{A}}

\newcommand{\lsc}{\text{LSC}}
\newcommand{\usc}{\text{USC}}
\newcommand{\Lomega}{L_{\omega_s}^1(\R^N)}

\newcommand{\dx}{\mathrm{d}x}
\newcommand{\dy}{\mathrm{d}y}
\newcommand{\dz}{\mathrm{d}z}

\newtheorem{theorem}{Theorem}
\newtheorem{lemma}[theorem]{Lemma}
\newtheorem{proposition}[theorem]{Proposition}

\newtheorem{conjecture}{Conjecture}
\newtheorem{definition}{Definition}
\newtheorem{remark}{Remark}

\begin{document}

		\title{\bf\Large The Landis conjecture for nonlocal elliptic operators: polynomial decay}
	\author[1]{Sebastián Flores Sepúlveda\footnote{sebastian.flores@dim.uchile.cl}}
	\author[2]{Gabrielle Nornberg\footnote{gnornberg@dim.uchile.cl}}
	\affil[1]{\small Center for Mathematical Modeling (CNRS IRL2807), University of Chile, Santiago, Chile and Institute for Applied Mathematics, University of Bonn, Germany.}
	\affil[2]{\small Department of Mathematics Engineering and Center for Mathematical Modeling (CNRS IRL2807), University of Chile, Santiago, Chile.}

	{\date{March 10, 2026}}

	\maketitle

	{\small\noindent{\bf{Abstract.}}
	We obtain a unique continuation result at infinity for fully nonlinear elliptic integro-differential operators of order $2s$ which satisfy the maximum and minimum principles in bounded subdomains, under the decay assumption $o(|x|^{-(N+2s)})$ at infinity.

	Our result is new even in the case of the fractional Laplacian in $\R^N$, as it unveils the nonlocal nature of the decay in Landis conjecture, evolving from exponential to polynomial.

	\smallskip

	}

	\medskip

\section{Introduction}

In \cite{kondratiev_qualitative_1988}, Landis and Kondratiev ask if an exponentially fast decaying solution to a linear Schrödinger equation in an exterior domain should be the trivial solution. This is known as the Landis conjecture and can be stated in the following manner:

\begin{conjecture}
Let $\Omega$ be either $\R^N$ or an exterior domain in $\R^N$, $V \in L^{\infty}(\R^N)$ with $\|V\|_{\infty} \leq 1$. If $u$ is a solution to 
\[\Delta u + Vu= 0 ~\text{ in }\Omega,\]
with $u(x) = O(e^{-\kappa |x|})$ as $|x| \rightarrow \infty$ for some $\kappa >1$, then $u \equiv 0$ in $\Omega$.
\end{conjecture}

A weaker version of this conjecture claims that if the solution $u$ decays faster than exponentially at infinity, that is, if $|u(x)| \leq Ce^{-|x|^{1+\varepsilon}}$ for some $C, \varepsilon >0$, then $u\equiv 0$. These questions are a kind of unique continuation principle at infinity. 

The conjecture was negatively  answered by Meshkov in \cite{meshkov_possible_1992}, where he showed a complex-valued potential and a nontrivial solution with decay $O(e^{-\kappa|x|^{4/3}})$ at infinity, for some $\kappa > 0$. He also proved that $4/3$ is optimal in the sense that if a solution behaves at infinity as $O(e^{-\kappa|x|^{4/3 + \varepsilon}})$ for some $\varepsilon>0$, it must be the trivial solution. The question for a real-valued potential has been widely studied in the last years. In \cite{logunov_landis_2020} the Landis conjecture in the plane was proven: if a solution decays at infinity with order $O(e^{-\kappa|x|(\log|x|)^{1/2}})$ then it must be trivial; see also \cite{davey_landis_2017, davey_landis_2020, kenig_landis_2015, kenig_quantitative_2015}.

In higher dimensions, the question remains open. It has been studied in \cite{bourgain_localization_2005} using Carleman estimates, and in \cite{rossi_landis_2020} with comparison arguments. In \cite{sirakov_vazquez_2021}, Sirakov and Souplet established a Landis-type result up to unbounded coefficients for fully nonlinear second order operators which satisfy the maximum principle in bounded subdomains. A sharp decay using criticality theory was addressed by Das and Pinchover in \cite{Pinchover}, in which polynomial decay appears in the presence of Hardy type decaying potentials.

The aim of this note is to address an analogous of the Landis conjecture for a class of integro-differential elliptic operators.  These operators are defined by 
\begin{equation}
\label{eq:linear}
Lu(x) = \PV \int_{\R^{N}} (u(x+y) - u(x)) K(y)\dy
\end{equation}
where $0<\frac{\lambda}{|y|^{N+2s}} \leq K(y) \leq \frac{\Lambda}{|y|^{N+2s}}$, for some fixed constants $s\in (0,1)$ and $0< \lambda \leq \Lambda$. In particular, when $\lambda = \Lambda = 1$ then $L$ is just the fractional Laplacian $(-\Delta)^s$  of order $2s$. For $s$, $\lambda$, $\Lambda$ fixed, we denote $\mathcal{L}$ to be the class of all operators $L$ in the form \eqref{eq:linear}. \smallskip

The Landis conjecture for the fractional Laplacian has been studied by Rüland and Wang in \cite{ruland_fractional_2019}, where they conclude a unique continuation principle at infinity for solutions of $(-\Delta)^su + Vu = 0$ under the assumption
\begin{equation}
\label{eq:integral_decay}
\int_{\R^N} e^{|x|^{\alpha}} |u(x)| \dx \leq C < \infty
\end{equation}
for some $\alpha>1$, whenever $V$ satisfies a regularity hypothesis given by $\sup_{x\in\R^N} |x \cdot D V(x) |\leq 1$. In \cite{kow_landis-type_2023}, Kow and Wang studied the question for $(-\Delta)^{1/2}u + b \cdot D u + Vu = 0$ in $\R^N$, and get a decay of $O(e^{-k|x|})$ under boundedness of the drift and potential terms, as well as $|D V|$. Note that this case is especially delicate since both $(-\Delta)^{1/2}$ and $b\cdot D u$ are of order 1. The proof of these results rely on the extension problem for the fractional Laplacian, together with the application of Carleman estimates. Recently, Das, Keller and Pinchover proved the Landis conjecture for fractional Schrödinger equations on graphs \cite{PinchoverGraphs}, under a polynomial decay of the solution, in the context of criticality theory. 

In this paper we extend the results in \cite{kow_landis-type_2023, ruland_fractional_2019} to fully nonlinear operators of order $2s$ which generalize  the class $\mathcal{L}$. 
We are able to relax condition \eqref{eq:integral_decay}, by allowing the solution to have polynomial decay, without assuming any decaying properties of the potential. 

We consider an operator $I$ of the form:
\begin{equation}
\label{eq:isaacs}
Iu(x) = \sup_{a\in \A} L_{a}u(x), \text{ or } Iu(x) = \inf_{a\in \A} L_{a}u(x),
\end{equation}
where $\A$ is a set and $L_{a} \in \mathcal{L}$ for every $a\in \A$. 
By solution we mean a viscosity solution in the sense of Definition \ref{def:viscosity} ahead. 
As it is customary in the theory of integral operators, those solutions belong to the space of measurable functions $u$ such that $\|u\|_{\Lomega} < \infty$, where
\[\|u\|_{\Lomega} := \int_{\R^N} \frac{|u(x)|}{1+|x|^{N+2s}} \dx \, ;\]  such a space is denoted 
by $\Lomega$.

Lastly, as in \cite{Pinchover, rossi_landis_2020, sirakov_vazquez_2021}, an assumption on the sign of the operator $\I +V$ will be pivotal. In the sequel, $\lambda_1^\pm ( \I + V, G) >0$  stands for the validity of both maximum and minimum principles for the operator $\I +V$ in the bounded $C^2$ subdomain $G$ of $\R^N $, see Definition \ref{defLambda+-}.  
\smallskip 

Within the vocabulary established, our main result reads as follows.

\begin{theorem}
\label{theo:landis}
Let $s\in (0,1)$,  $I$ be an operator as in \eqref{eq:isaacs}, and $V:\R^N \rightarrow \R$ be a bounded continuous function.
Assume $\lambda_1^\pm ( I + V, G) >0$ for each bounded $C^2$ subdomain $G$ of $\R^N$. \smallskip
	 If $u\in \Lomega$ is a viscosity solution to $Iu + Vu = 0$ in $\R^N$, such that 
\begin{equation}\label{H}
\lim_{|x| \rightarrow \infty} |u(x)||x|^{N+2s} = 0,
\end{equation}
then $u\equiv 0$ in $\R^N$.
\end{theorem}

\begin{remark}
We have stated Theorem \ref{theo:landis} for the whole space $\Omega=\R^N$ for ease of notation, but under minor modifications its proof may also address the case when $\Omega$ is an exterior domain, as in \cite{sirakov_vazquez_2021}, by asking additionally that $u$ has a sign in $\R^N \setminus \Omega$ in order to conclude that $u\equiv 0$ in $\Omega$.
\end{remark}

Theorem \ref{theo:landis} unveils the difference of decay rate of the Landis conjecture, evolving from exponential to polynomial in the presence of a nonlocal operator. This result is new even in the context of the fractional Laplacian, and it is consistent with the polynomial decay rate of the nonlocal nonlinear Schrödinger equation in \cite[Theorem 1.5]{Felmer_Quaas_Tan_2012}, see also \cite[Chapter 7]{BucurValdinoci}. 

Hypothesis \eqref{H} seems to be sharp.
For instance, regarding the fractional Laplacian,
Lemma 2.1 in \cite{BoVa} says that a $C^2$ function $u$ defined in the whole space, which is positive radial and decreasing for $|x|\ge 1$, and  $u(x)\approx |x|^{-(N+2s)}$ for large $|x|$, satisfies $ |(-\Delta)^s u| \approx |u| $ in a neighborhood of infinity.
Besides,  when $V\equiv -\frac12$, Lemma 4.3 in \cite{Felmer_Quaas_Tan_2012} shows the existence of a classical solution $w$ to problem 
$(-\Delta)^sw + \frac12 w = 0 \text{ in the exterior of the ball }  B_1$
such that 
$0 <w(x) \leq \frac{c_0}{|x|^{N+2s}}$
for some $c_0 > 0$.

\smallskip

The remaining of the text splits as follows.
In Section \ref{sect:def} we recall some definitions and auxiliary tools, while in Section \ref{sect:main} we give the proof of Theorem \ref{theo:landis}. 
The latter relies on a nonlocal weak Harnack estimate by Ros-Oton and Serra \cite{ros-oton_boundary_2019}, properly adapted for equations with zero order terms, which in turn gives a lower decay bound for positive supersolutions. This approach can be seen as a kind of nonlocal counterpart of techniques in \cite{sirakov_vazquez_2021}, slightly improved due to the nonlocal character of the estimates.

\section{Auxiliary results}
\label{sect:def}

Throughout this paper, $s\in (0,1)$, $0<\lambda \leq \Lambda$ will be fixed quantities. First we recall the definition of viscosity solution, as given in \cite{fernandez-real_integro-differential_2024}. 

\begin{definition}[Viscosity solution]
\label{def:viscosity}
Let $\Omega \subset \R^N$ be any open set, $s\in (0,1)$ and $f,c\in C(\Omega)$, and consider the equation
\begin{align}
\label{eq:general}
Iu + V(x)u &= f(x)
\end{align}
\begin{enumerate}[i.]
\item We say that $u \in \text{USC}(\Omega) \cap \Lomega$ is a \textit{viscosity subsolution} to \eqref{eq:general} if for any $x\in \Omega$ and any neighborhood of $x$ in $\Omega$, $N_x$, and for any test function $\psi \in \Lomega$ such that $\psi \in C^2(N_x)$, $\psi(x) = u(x)$ and $\psi \geq u$ in $\R^N$ we have that $I\psi(x) + V(x)u(x) \geq f(x)$. We denote this by $Iu(x) + V(x)u(x) \geq f(x)$.
\item We say that $u \in \lsc(\Omega) \cap \Lomega$ is a \textit{viscosity supersolution} to \eqref{eq:general} if for any $x\in \Omega$ and any neighborhood of $x$ in $\Omega$, $N_x$, and for any test function $\phi \in \Lomega$ such that $\phi \in C^2(N_x)$, $\phi(x) = u(x)$ and $\phi \leq u$ in $\R^N$ we have that $I\phi(x) + V(x)u(x) \leq f(x)$. We denote this by $Iu(x) + V(x)u(x) \leq f(x)$.
\item A function $u \in C(\Omega) \cap \Lomega$ is a \textit{viscosity solution} to \eqref{eq:general} if it is both a viscosity subsolution and viscosity supersolution.
\end{enumerate}
\end{definition}

We define the Pucci extremal operators associated to $\mathcal{L}$ by
\begin{align*}
\M^+u(x) = \sup_{L\in \mathcal{L}} Lu(x), \quad 
\M^-u(x) = \inf_{L\in \mathcal{L}} Lu(x),
\end{align*}
and note that, for any operator $I$ of the form \eqref{eq:isaacs} and any smooth bounded functions $u$, $v$, we have
\begin{equation*}
\label{eq:ellipticity}
\M^-(u-v) \leq Iu - Iv \leq \M^+(u-v).
\end{equation*}
This inequality also holds for viscosity solutions, accordingly to the next lemma.
\begin{lemma}[Proposition 3.2.15 in \cite{fernandez-real_integro-differential_2024}]
\label{lemma:eqdifference}
Let $I$ be as in \eqref{eq:isaacs} and $\Omega\subset\R^N$ be an open set. If $u\in\Lomega$ is lower semicontinuous, $v\in\Lomega$ is upper semicontinuous, $f,g \in C(\Omega)$, and $Iu \leq f$, $Iv \geq g$ holds in $\Omega$ in the viscosity sense, then $\M^-(u-v) \leq f-g,\text{ in } \Omega$ also holds in the viscosity sense.
\end{lemma}

\begin{remark}\label{rmk-lemma2}
The conclusion of Lemma \ref{lemma:eqdifference} is also true if $\M^-$ is replaced by $\M^-_{\mathcal{A}}:=\inf_{a\in \mathcal{A}} L_a$, since the proof of Proposition 3.2.15 in \cite{fernandez-real_integro-differential_2024} follows exactly the same by using $\M^-_{\mathcal{A}}\le  Iu -Iv$ for smooth functions $u,v$.
\end{remark}

Note that if a function $u \in \Lomega$ is of class $C^{1,1}$ in a neighborhood of a point $x \in \R^N$, then $Lu(x)$ is well defined. On the other hand, the next lemma says that, if $u$ is nonsmooth but is a viscosity supersolution (subsolution) to $Iu=f$, we only need for $u$ to be touched from below (above) by some test function at $x$.
\begin{lemma}[Pointwise evaluation, Lemma 3.3 in \cite{caffarelli_regularity_2009}]
\label{lemma:one-sided-reg}
Let $\Omega \subset \R^N$ be an open set, $f\in C(\Omega)$, and $I$ an operator of the type \eqref{eq:isaacs}. If $u \in \lsc(\Omega) \cap \Lomega$ is a viscosity solution to 
$ I u \leq f $
and there exists a test function $\psi$ touching $u$ from below at $x$, then for each $L\in \mathcal{L}$, $Lu(x)$ is well defined and $Iu(x) \leq f(x)$.
\end{lemma}

The stability of viscosity sub and supersolutions under suitable limits is given by the following lemma.
\begin{lemma}[Proposition 3.2.12 in \cite{fernandez-real_integro-differential_2024}]
\label{lemma:stability}
Let $\Omega \subset \R^N$ be an open set, $\{f_k\}_{k\in \N}\subset C(\Omega)$, $f\in C(\Omega)$ be such that $f_k \rightarrow f$ locally uniformly in $\Omega$, and $I$ an elliptic operator of the form \eqref{eq:isaacs}. Let $u_k\in C(\Omega)\cap\Lomega$ be a sequence of functions such that 
\begin{enumerate}
\item $I u_k \leq f_k$ in the viscosity sense in $\Omega$,
\item $u_k \rightarrow u$ locally uniformly in $\Omega$,
\item $u_k \rightarrow u$ in $\Lomega$,
\end{enumerate} 
then $I u\leq f$ in the viscosity sense.
\end{lemma}

We will also need the following result on the solvability of the Dirichlet problem in bounded domains. It is a special case of Corollary 5.7 in \cite{mou_perrons_2017}.
\begin{proposition}
\label{prop:existence_isaacs}
Let $G \subset \R^N$ be a bounded $C^2$ domain, $\I$ as in \eqref{eq:isaacs}, $f, V:G \rightarrow \R$ continuous, and assume $V \leq 0$. Then the problem
\[\begin{cases}
\I u + Vu = f & \text{in } G \\
u = g & \text{in } \R^N \setminus G 
\end{cases}\]
admits a viscosity solution $u\in C(\overline{G}) \cap \Lomega$ for any $g:\R^N \rightarrow \R$ bounded and continuous.
\end{proposition}

%%%%%%%%%%%%%%%%%%
We state now the minimum principle for viscosity supersolutions. It is an extension of Lemma 3.2.19 of \cite{fernandez-real_integro-differential_2024}, when we consider a nonpositive continuous potential.
\begin{lemma}\label{mP}
Let $s \in(0,1), V\in C(\Omega)$ with $V \leq 0$, and let $\Omega \subset \mathbb{R}^N$ be any bounded open set, and let $u \in L S C(\bar{\Omega}) \cap L_{\omega_s}^1\left(\mathbb{R}^N\right)$ be a viscosity solution to
$$
\begin{cases}\mathcal{M}^{-}[u]+V u \leq 0 & \text { in } \Omega \\ u \geq 0 & \text { in } \mathbb{R}^N \backslash \Omega . \end{cases}
$$
Then $u \geq 0$ in $\mathbb{R}^N$.
\end{lemma}

\begin{proof}
By the lower semicontinuity of $u$, there exists $x_0 \in \overline{\Omega}$ such that $u$ reaches its minimum at $x_0$. By contradiction, suppose $u\left(x_0\right)=-m<0$. Since $u \geq 0$ on $\partial\Omega$, then $x_0 \in \Omega$.
	
Defining $\psi=-m \chi_{\Omega}$, we have that $\psi$ is a smooth function in a neighborhood of $x_0$ contained in $\Omega$, and satisfies $u \geq \psi$ in $\mathbb{R}^N$. By the definition of viscosity solution, we get
\begin{center}
		$
	\mathcal{M}^{-} \psi\left(x_0\right)+V (x_0) u\left(x_0\right) \leq 0 \Longrightarrow \mathcal{M}^{-} \psi\left(x_0\right) \leq-V(x_0) u\left(x_0\right) \leq 0 .
	$
\end{center}

	But	since $x_0$ is a minimum for $u$ in $\overline\Omega$,
	$$
	\begin{aligned}
	\textstyle	\mathcal{M}^{-} \psi\left(x_0\right) &  	\textstyle = \inf _{L \in \mathcal{L}} \int_{\mathbb{R}^N}\left(u(y)-u\left(x_0\right)\right) K\left(y-x_0\right) \mathrm{d} y 	\textstyle \ge \inf _{K} \int_{\mathbb{R}^N \backslash \Omega} m K\left(y-x_0\right) \mathrm{d} y \\
		& 	\textstyle \geq \int_{\mathbb{R}^N \backslash \Omega} \frac{\lambda m}{\left|y-x_0\right|^{N+2 s}} \mathrm{~d} y  >0,
	\end{aligned}
	$$
which is a contradiction. Hence $u \geq 0$ in $\mathbb{R}^N$.
\end{proof}
%%%%%%%%%%%%%%%%%%

Throughout the text, a key half-Harnack inequality based on the following proposition will come into play. Proposition \ref{prop:weakharnack} below is essentially that of Theorem 2.2 in \cite{ros-oton_boundary_2019}, with some minor modifications to account for the zero order term.

\begin{proposition}[Weak Harnack inequality]
\label{prop:weakharnack}
Let $V \in C(B_1)\cap L^{\infty}(\R^N)$, with $\|V\|_{L^{\infty}(B_1)} \leq 1$, let $u\in LSC(B_1) \cap L^1_{\omega_s}(\R^N)$ be a viscosity solution to 
\[\begin{cases}
\M^-u + Vu \leq C_0 & \text{in }B_1\\
u\geq 0 & \text{in } \R^N
\end{cases}\]
with $C_0 > 0$. Then there exists $C=C(N,s,\Lambda, \lambda)>0$ such that
\[\|u\|_{\Lomega}\leq C(\inf_{B_{1/2}}u + C_{0}) \]
\end{proposition}

\begin{proof}
Let $\eta\in C_c^{\infty}(B_{3/4})$, $0\leq \eta \leq 1$ and $\eta\equiv 1$ in $B_{1/2}$. Now, $\inf_{B_{3/4}} u \geq 0$, so for $x\in B_{3/4}$
\[u(x) \geq \inf_{B_{3/4}} u \geq \eta \inf_{B_{3/4}} u\]
and, as $\eta \equiv 0$ outside $B_{3/4}$, we actually have $u\geq \eta \inf_{B_{3/4}} u$ in $\R^N$. We can then define
\[t:=\sup\{\tau \in \R:~ u(x) \geq \tau\eta(x),~\forall x\in \R^N\}\ge \inf_{B_{3/4}} u . \]
Note that by taking $x\in B_{1/2}$, we have $u(x)\geq t$, so that $\inf_{B_{1/2}}u \geq t$.

By the lower semiconinuity of $u - t\eta$, there exists $x_0\in B_{3/4}$ such that $\inf_{B_{3/4}}(u-t\eta) = (u-t\eta)(x_0)\geq 0$. If we had $u(x_0)-t \eta(x_0) > 0$, we would get $u > t \eta $ in $B_{3/4}$, so there exists $\varepsilon >0$ such that $u>(t+ \varepsilon)\eta$ in $B_{3/4}$. Since $\eta \equiv 0$ outside $B_{3/4}$, $u\geq 0 = (t+\varepsilon)\eta$. This means $t + \varepsilon >t$ satisfies $u\geq (t+\varepsilon)\eta$ in  $\R^N$, contradicting the maximality of $t$. Thus, we must have $u(x_0) - t\eta(x_0) = 0$.

This means that $t\eta$ touches $u$ from below at $x_0$, then by Lemma \ref{lemma:one-sided-reg}, $\M^-u(x_0)$ is well defined and we have
\[\M^-u(x_0) +V(x_0) u(x_0) \leq C_0\]
Since $u-t\eta$ can be touched from below by a constant function, then $\M^-[u-t\eta](x_0)$ is also classically defined.

Note that
$$\begin{aligned} 
	\mathcal{M}^{-}[\eta]\left(x_0\right) 
	& =\inf _{K} \int_{\mathbb{R}^N}\left(\eta(y)-\eta\left(x_0\right)\right) K\left(x_0-y\right) \mathrm{d} y \\ 
	& \geq \inf _{K} 
	\left\{
	\int_{ B_\delta (x_0)}
	(\eta(y)- \eta( x_0 )) K\left(x_0-y\right) \mathrm{d} y 
	-   \int_{ \mathbb{R}^N  \setminus B_\delta (x_0) }
	K\left(x_0-y\right) \mathrm{d} y 
	\right\} \\ 
	& 
	\ge c_0 
	\int_{B_\delta (x_0)} 
	\frac{  \mathrm{~d} y }{\left|x_0-y\right|^{N-2+2 s}}
	- \Lambda \int_{ \mathbb{R}^N \setminus B_\delta (x_0) } 
	\frac{\mathrm{~d} y}{\left|x_0-y\right|^{N+2 s}} 
	\geq-c_1,
\end{aligned}$$
with $c_1 =c_1 ( N,s,\Lambda)$ because $ 0 \leq \eta \leq 1 $. Indeed, since $\eta $ is a $C^2 $ function, we write
$\eta (y)=\eta (x_0)+D\eta (x_0) \cdot (y-x_0)+O(|y-x_0|^2)$ in $B_\delta (x_0)$, for some fixed $\delta >0$, and use
the symmetry of  the kernel in $B_\delta (x_0)$ to cancel the derivative term.

Now, since $(u-t\eta)(x_0)=0$, we get
\begin{align*}
	\M^-[u-t\eta](x_0) \leq &\; \M^-[u](x_0) - t\M^-[\eta](x_0) +V(x_0)u(x_0) - tV(x_0)\eta (x_0) \\
	\leq &\; C_0 - t\M^-[\eta](x_0) - tV(x_0)\eta (x_0)\\
	\leq &\; C_0 + tc_1 + t
\end{align*}
because $V(x_0) \geq -1$ and $\eta \leq 1$.

On the other hand, using $ u-t\eta\ge(u-t\eta)(x_0)=0$,  a lower bound is produced:
\begin{align*}
\M^-[u-t\eta](x_0) &\geq \lambda \int_{\R^N} \frac{u(z) - t\eta(z)}{|z-x_0|^{N+2s}} \dz \\
&\geq \frac{\lambda}{2^{N+2s}} \left(\int_{\R^N} \frac{u(z)}{1+ |z|^{N+2s}}\dz - t\|1\|_{\Lomega} \right) \\
&= \Tilde{C}\|u\|_{\Lomega} - tc_2
\end{align*}
where $\Tilde{C} :=\frac{\lambda}{2^{N+2s}} $ and $c_2 := \Tilde{C}\|1\|_{\Lomega} $ depend only on $N$, $s$ and $\lambda$. 

Combining the above inequalities, it yields
\begin{align*}
\Tilde{C}\|u\|_{\Lomega} - tc_2 &\leq C_0 + tc_1 + t\\
\Longrightarrow~~ \Tilde{C}\|u\|_{\Lomega} &\leq C_0 + t (1 + c_1 + c_2)\\
\Longrightarrow \|u\|_{\Lomega} &\leq C (\underset{B_{1/2}}{\inf}u + C_0),
\end{align*}
using the fact that $t\leq \underset{B_{1/2}}{\inf}u$ and by setting $C:= \frac{1 + c_1 + c_2}{\Tilde{C}}$.
\end{proof}

We also recall a global regularity result from \cite{quaas_principal_2020}, followed by a Harnack inequality from  \cite{Davila-Quaas-Topp-parab}.

\begin{proposition}[Theorem 3.6 in \cite{quaas_principal_2020}]
\label{prop:regularity}
Let $G$ be a bounded $C^2$ domain, $s\in (0,1)$ and $f\in L^{\infty}(G)$. If $u\in C(G) \cap L^1_{\omega_s}(\R^N)$ is a viscosity solution of 
\[\begin{cases}
-\I [u] = f & \text{in }G\\
u=0 & \text{in } \R^N\setminus G,
\end{cases}\]
then $u\in C^{\alpha}(\overline{G})$ for some $\alpha \in (0,1)$.
\end{proposition}

\begin{proposition}\label{harnack}
	Let $\Omega \subset \mathbb{R}^N$ be a bounded domain, $s\in (0,1)$, and $V\in C(\overline{\Omega})$. Then for any open set $\tilde{\Omega}$ with $\tilde{\Omega} \subset \subset \Omega$ there exists a constant	$C>0$ depending on $N$, $s$, $\lambda$, $\Lambda$, $\Omega$, $\tilde{\Omega} $, and	$\|V\|_{ L^{\infty} ( \Omega ) } $, such that for any viscosity solution $u\in C( \Omega ) \cap \Lomega$ to
	\begin{align*}
		\begin{cases}
			\M^-u + Vu \leq C_0 & \text{in }\Omega\\
			\M^+u + Vu \geq -C_0 & \text{in }\Omega\\
			u\geq 0 & \text{in } \R^N
		\end{cases}
	\end{align*}
	we have
	$
	\sup _{\tilde{\Omega}} u\leq C  (\inf _{\tilde{\Omega}} u+ C_0 ).
	$
\end{proposition}
The proof of Proposition \ref{harnack} is exactly the same as the proof of Theorem 3.2  in \cite{Davila-Quaas-Topp-parab} (see also Theorem 1.3 in there), since $s$ can be any number between $0$ and $1$ in the absence of drift terms. We also give a simpler proof in the appendix in the case where the class $\mathcal{A}$ consists of linear operators with homogeneous kernels of degree $-N-2s$ generalizing the fractional Laplacian.

Now we highlight the definition of maximum/minimum principle holding in a bounded domain. In the sequel we consider an operator $F=\I+V$, where $\I$ has the form \eqref{eq:isaacs}.

\begin{definition}\label{defLambda+-}
\label{def:mp_operator}
We say the operator $F$ satisfies the maximum principle (resp. minimum principle) in the bounded $C^2$ domain $G$, and we denote it by $\lambda_1^+(F,G)>0$ (resp. $\lambda_1^-(F,G)>0$), if for every function $v\in \usc(\Omega)\cap \Lomega$ $(v\in \lsc(\Omega)\cap \Lomega )$ that satisfies \begin{center}
	$F[v]\geq 0$ ($F[v]\leq 0$) in $\Omega$ in the viscosity sense, with $v \leq 0$ ($v\ge 0$) outside $G$,
\end{center} one has $v\leq 0$ (resp. $v\ge 0$)  in $\R^N$.
\end{definition}

Our notation on the validity of the maximum and minimum principles in terms of the first eigenvalue quantities $\lambda_1^\pm(F,G)$ comes from the equivalences proved in \cite{quaas_principal_2020}. See those equivalences and analogous definitions in \cite{pq1} for the local case, in which 
$$\lambda_1^-(\M^-, G)=\lambda_1^+(\M^+, G)\le \lambda_1^-(\M^+, G)=\lambda_1^+(\M^-, G).$$

\section{Main results}
\label{sect:main}

In this section we turn to prove a weak Harnack inequality in balls of arbitrarily large radius. Special care must be taken on the dependence of the constant on the radius $R$, as this will ultimately give us the lower polynomial bound on the decay of positive supersolutions.

\begin{theorem}
\label{theo:harnack}
Let $R\geq 1$, $s\in (0,1)$ and $V\in C(B_{R+1})\cap L^{\infty}(B_{R+1})$.
There exists a positive constant $C = C(N, s, \lambda, \Lambda)$ such that for any viscosity solution $u\in C(\overline{B_{R+1}}) \cap \Lomega$ to
\begin{align*}
\begin{cases}
\M^-u + Vu \leq C_0 & \text{in }B_{R+1} \\
u\geq 0 & \text{in } \R^N
\end{cases}
\end{align*}
we have
\begin{align}
\label{eq:sharpharnack}
\|u\|_{\Lomega} \leq C(1+\|V\|_{L^{\infty}(B_{R+1})}) (1 + R^{N+2s})\left(\underset{B_R}{\inf}~u + C_0 \right).
\end{align}
\end{theorem}

In order to prove Theorem \ref{theo:harnack}, we first establish a key lemma which compares the $\Lomega$-norm of a function with its scaled-translated version.
\begin{lemma}
	\label{lemma:bounds}
	Let $u \in \Lomega$, $x_0 \in \R^N$ and $r_0 \in (0, 1]$. If one defines $v(y) = u(x_0 + r_0 y)$, then there exists $C$ that depends on $N$ and $s$, such that
	\[\frac{C^{-1}r_0^{2s}}{1+|x_0|^{N+2s}}\|u\|_{\Lomega} \leq \|v\|_{\Lomega} \leq  \frac{C}{r_0^{N}} (1+|x_0|^{N+2s}) \|u\|_{\Lomega}.\]
\end{lemma}

\begin{proof}
	Note that
	\[\|v\|_{\Lomega} = r_0^{2s} \int_{\R^N} \frac{1 + |z|^{N+2s}}{r_0^{N+2s} + |z-x_0|^{N+2s}} \frac{|u(z)|}{1+|z|^{N+2s}}\dz.\]
	Thus, it suffices to prove
	\[\frac{3^{-N-2s}}{1+|x_0|^{N+2s}} \leq \frac{1 + |z|^{N+2s}}{r_0^{N+2s} + |z-x_0|^{N+2s}} \leq \left( \frac{2}{r_0}\right)^{N+2s} (1+|x_0|^{N+2s}).\]
	
	If $|z| \leq 2|x_0|$,
	\[\frac{1 + |z|^{N+2s}}{r_0^{N+2s} + |z - x_0|^{N+2s}} \leq \frac{1 + 2^{N+2s}|x_0|^{N+2s}}{r_0^{N+2s}} \leq \left( \frac{2}{r_0}\right)^{N+2s}(1+|x_0|^{N+2s}),\]
	and, as $r_0\leq 1$ and $|z-x_0| \leq |z| + |x_0| \leq 3|x_0|$,
	\[\frac{1 + |z|^{N+2s}}{r_0^{N+2s} + |z - x_0|^{N+2s}} \geq \frac{1 + |z|^{N+2s}}{1 + 3^{N+2s} |x_0|^{N+2s}} \geq \frac{1}{3^{N+2s}}\frac{1}{1+|x_0|^{N+2s}}.\]
	
	If $|z| > 2|x_0|$, we have, by the triangle inequality,
	\[|z-x_0| \geq |z| - |x_0| \geq |z| - \frac{|z|}{2} = \frac{|z|}{2} \geq r_0 \frac{|z|}{2},\]
	which gives
	\[\frac{1 + |z|^{N+2s}}{r_0^{N+2s} + |z - x_0|^{N+2s}} \leq  \frac{1 + |z|^{N+2s}}{r_0^{N + 2s} (1 + (|z|/2)^{N+2s})} \leq \left( \frac{2}{r_0}\right)^{N+2s} \leq \left( \frac{2}{r_0}\right)^{N+2s} ( 1+ |x_0|^{N+2s}).\]
	
	On the other hand, $|z-x_0| \leq |z| + |x_0| \leq 2|z|$. This implies
	\[\frac{1 + |z|^{N+2s}}{r_0^{N+2s} + |z - x_0|^{N+2s}} \geq \frac{1+|z|^{N+2s}}{1+2^{N+2s}|z|^{N+2s}} \geq \frac{1}{2^{N+2s}} \geq \frac{1}{2^{N+2s}} \frac{1}{1+ |x_0|^{N+2s}}.\]
\end{proof}

\begin{proof}[Proof of Theorem \ref{theo:harnack}]
	The first step is to derive a weak Harnack inequality in small balls.

Let $x_0 \in \R^N$, and
\begin{equation}
\label{eq:r0}
0 < r_0 \leq (1 + \|V\|_{L^{\infty}(B_{2r_0}(x_0))})^{-1/2s},
\end{equation}
and let $u \in \lsc(B_{r_0}(x_0)) \cap \Lomega$ a viscosity solution to 
\[\begin{cases}
\M^-u + Vu \leq C_0 & \text{in }B_{r_0}(x_0) \\
u\geq 0 & \text{in } \R^N. 
\end{cases}\]
We claim that
\begin{equation}
\label{eq:wh_smallballs}
\|u\|_{\Lomega} \leq C r_0^{-2s} (1+|x_0|^{N+2s}) \left(\inf_{B_{r_0/2} (x_0)}u + r_0^{2s}C_0 \right),
\end{equation}
for some $C=C(N,s,\lambda,\Lambda)>0$.

Indeed, if we define 
\begin{equation}
\label{eq:v}
v(y) = u(x_0 + r_0y),~\tilde{V}(y) = r_0^{2s}V(x_0+r_0 y),
\end{equation} 
we get $\M^-v(y) = r_0^{2s}\M^-u(x_0 +r_0y)$, and so $v$ satisfies  
\[\begin{cases}
\M^-v + \tilde{V}v \leq r_0^{2s}C_0 & \text{in }B_1 \\
v\geq 0 & \text{in } \R^N
\end{cases}\]
in the viscosity sense. Furthermore, 
\begin{align*}
\|\tilde{V}\|_{L^{\infty}(B_1)} =r_0^{2s} \|V\|_{L^{\infty}(B_{r_0}(x_0))} &\leq \frac{\|V\|_{L^{\infty}(B_{r_0}(x_0))}}{1 + \|V\|_{L^{\infty}(B_{2r_0}(x_0))}} \leq 1
\end{align*}
thus by Proposition \ref{prop:weakharnack}, we obtain 
\[\|v\|_{\Lomega}\leq C(\inf_{B_{1/2}}v + r_0^{2s}C_0). \]
In view of Lemma \ref{lemma:bounds}, we get 
\[\|u\|_{\Lomega} \leq C r_0^{-2s} (1 + |x_0|^{N+2s})   \|v\|_{\Lomega} ,
 \]
which gives us \eqref{eq:wh_smallballs}.

\smallskip

Next, we set $r_0 = (1 + \|V\|_{\infty})^{-\frac{1}{2s}} \leq 1$. 
By continuity of $u$ in $B_{R+1}$, $u$ achieves a minimum and a maximum in $\overline{B}_R$, so we can take $x_*\in {B}_R$ such that
\begin{equation}
\label{eq:coupdegrace}
\inf_{B_{r_0/2}(x_*)}u \leq  \inf_{B_R}u .
\end{equation}

Applying \eqref{eq:wh_smallballs} at $x_*$ we obtain
\begin{align*}
\|u\|_{\Lomega} & \leq C(1 + \|V\|_{\infty}) (1+|x_*|^{N+2s})\left(\inf_{B_{r_0/2} (x_*)}u  + r_0^{2s}C_0\right) \\
&  \leq C(1 + \|V\|_{\infty})  (1 + R^{N+2s})\left(\inf_{B_{r_0/2} (x_*)}u + C_0\right).
\end{align*}
Combining this with \eqref{eq:coupdegrace} gives \eqref{eq:sharpharnack}, for  $R\geq 1$.
\end{proof}

The proof of Theorem \ref{theo:harnack} gives a short proof of the strong minimum principle, and more generally assuming the minimum principle holds we have the following result:
\begin{lemma}
\label{lemma:smp}
Let $G \subset \R^N$ be a bounded $C^2$ domain, and $F=\I+V$, where $\I$ has the form \eqref{eq:isaacs} and $V\in L^{\infty}(G)$, such that $\lambda_1^- (F, G)>0$, $s\in (0, 1)$. If $u\in \lsc(\overline{G}) \cap L_{\omega_s}^1(\R^N)$ is a viscosity solution to
\[\begin{cases}
F[u] \leq 0 & \text{ in }G\\
u\geq 0 & \text{ in }\R^N \setminus G, 
\end{cases}\]
then either $u>0$  or $u\equiv 0$ in $G$.
\end{lemma}

\begin{proof}
We first apply the minimum principle to get $u\ge 0$ in $\R^N$. Next, if $u(x_0)= 0$ for some $x_0 \in G$, we take $r_0> 0$ such that $B_{r_0}(x_0) \subset G$ and such that \eqref{eq:r0} holds. We define $v$ as in \eqref{eq:v}, then inequality \eqref{eq:wh_smallballs} holds with $C_0 = 0$. This implies $\|u\|_{\Lomega} = 0$, so $u\equiv 0$.
\end{proof}

We need a viscosity supersolution which is positive in the whole space in order to apply a comparison principle and then conclude the unique continuation result at infinity. The next proposition is a nonlocal version to Proposition 4.2 in \cite{sirakov_vazquez_2021}, and will give us a positive bound from above to our solution.

\begin{proposition}
\label{prop:positivesol}
Let $V \in C(\R^N) \cap L^{\infty}(\R^N)$ be such that $\lambda_1^\pm ( I + V, G) >0$ for each bounded $C^2$ subdomain $G$ of $\R^N$. Then there exists $\psi \in C(\R^N) \cap \Lomega$ a viscosity solution to
\[I \psi + V\psi = 0 \text{ in } \R^N \]
with $\psi>0$ in $\R^N$.
\end{proposition}

The proof of Proposition \ref{prop:positivesol} relies on the following Leray-Schauder alternative:
\begin{proposition}[Corollary 1.19 in \cite{bandle_solutions_2004}]\label{LS}
	Let $X$ be a Banach space, and $S:X\rightarrow X$ a compact operator. Then the following alternative holds:
	\begin{enumerate}
		\item either $x-tS(x) = 0$ has a solution for every $t\in [0,1]$;
		\item or the set $\{x\in X: \exists t \in [0,1]:~x-tS(x) = 0\}$ is unbounded.
	\end{enumerate}
\end{proposition}

\begin{proof}[Prooof of Proposition \ref{prop:positivesol}]
	Let $F:=\I +V$,  $G\subset \R^N$ be a $C^2$ bounded domain, and $g \in C(\overline{G})$. Since we are assuming $V\in C(\R^N) \cap L^{\infty}(\R^N)$, there exists $\sigma >0$ such that $V-\sigma \leq 0$ in $G$. For any $v\in C(\overline{G})$, let $u \in C(\overline{G})$ be the unique viscosity solution to
	\begin{equation}\label{D-FP}
\textrm{		$F[u] -\sigma u= g- \sigma v$ in $G$ and $u = 0$ in $\R^N \setminus G$. }
	\end{equation}
Indeed, existence is ensured by Proposition \ref{prop:existence_isaacs}. 
For the uniqueness, if $u_1, u_2$ are two viscosity solutions of \eqref{D-FP}, we use Lemma \ref{lemma:eqdifference} to get 
\[\M^{-}(u_i - u_j) \leq -(V - \sigma)(u_i - u_j),\]
for $ i,j = 1,2$, and then apply the minimum principle (Lemma \ref{mP}) for the coercive extremal operator $\M^- + (V -\sigma) $ to ensure that $u_1=u_2$ in $\R^N$.
	So we can define an operator $S:C(\overline{G}) \rightarrow C(\overline{G})$ such that $S[v] = u$ is the solution to \eqref{D-FP}.

Now, taking $C_{0} = \|g\|_{C(\overline{G})} + \sigma\|v\|_{C(\overline{G})} + \|V- \sigma\|_{C(\overline{G})} \|u\|_{C(\overline{G})} $, we get
\[\begin{cases}
\M^+u \geq -C_{0} & \text{ in }G \\
\M^-u \leq C_{0} & \text{ in }G \\
u= 0 & \text{ in }\R^N\setminus G
\end{cases}\]
so, by Proposition \ref{prop:regularity}, there exists $\alpha \in (0,1)$, depending only on $G$, $N$, $s$, $\lambda$ and $\Lambda$, such that $u \in C^{ \alpha}(\overline{G})$. This implies that, if $A\subset C(\overline{G})$ is a bounded set, then $S(A)$ is bounded in $C^{\alpha}(\overline{\Omega})$ and, by the compact inclusion of $C^{\alpha}(\overline{\Omega})$ in $C(\overline{\Omega})$, it is a relatively compact set in $C(\overline{G})$. This means that $S: C(\overline{G}) \rightarrow C(\overline{G})$ is a compact operator. 

We will apply the Leray-Schauder alternative (Proposition \ref{LS}) to conclude that $S$ has a fixed point.

Aiming for a contradiction, let us suppose that there exist $(u_n)_{n\in \N} \subset C(\overline{G})$ and $(t_n)_{n\in \N} \subset [0,1]$ such that $u_n - t_n S[u_n] = 0$ for every $n\in \N$, and $\|u_n\|_{C(\overline{G})} \rightarrow \infty$. Note that by compacity we can assume that $t_n \rightarrow t\in [0,1]$. Now, by definition of $S$, it holds that
\begin{align*}
F\left[u_n / t_n\right]-\sigma\left(u_n / t_n\right)&=g-\sigma u_n \\
\Rightarrow F[u_n] - \sigma u_n &= t_n (g - \sigma u_n) .
\end{align*}
By defining $\Tilde{u}_n = \frac{u_n}{\|u_n\|_{C(\overline{G})}}$, we get $\|\Tilde{u}_n\|_{C(\overline{G})} = 1$ and also
\[F[\Tilde{u}_n] - \sigma \Tilde{u}_n = t_n \left( \frac{g}{\|u_n \|_{C(\overline{G})}} - \sigma\Tilde{u}_n\right)\]
which means, again by Holder regularity, that
\[\|\tilde u_n \|_{C^{\alpha}(\overline{G})} \leq C(1 + \|u_n \|^{-1}_{C(\overline{G})} \|g\|_{C(\overline{G})}) \leq C.\]
Therefore, $(\tilde{u}_n)_{n\in \N}$ is bounded in $C^{\alpha}(\overline{G})$, which implies that there exists a convergent subsequence (which will not be relabeled) $\Tilde{u}_n \rightarrow \Tilde{u} \in C(\overline{G})$ with $\Tilde{u}$ a viscosity solution to
\[F[\Tilde{u}] - \sigma (1-t) \Tilde{u} = 0.\]
Hence $\Tilde{u} \equiv 0$, because we know that if $F$ satisfies the maximum principle, then $F-\sigma (1-t) $ also does for $\sigma (1-t) \geq 0$. But this contradicts the fact that $\|\Tilde{u}\|_{C(\overline{G})} =1$.

We conclude that for every $t\in [0,1]$, there exists a viscosity solution $u$ for $u-tS[u] = 0$ in $G$, $u=0$ in $\R^N\setminus G$. Take $g = -V \in C(\overline{G})$ and $t=1$ and we have found a viscosity solution to
\[\begin{cases}
F[v] = -V &\text{ in } G\\
v= 0 & \text{ in } \R^N\setminus G.
\end{cases}\]
By setting $u= v + 1$ we derive
\[F[u] = \I[v + 1] + Vv + V = F[v] +V = 0 \text{ in }G\]
and $u = 1$ outside of $G$. By Lemma \ref{lemma:smp} we have $u>0$ in $G$.

We claim that this solution $u$ is unique, due to our hypothesis $\lambda_1^\pm (I+V,G)>0$. 
Indeed, let $u_1, u_2$ be viscosity solutions of $F[v_i]=0$ in $G$, with $u_i =1$ in $\R^N\setminus G$, $i=1,2$, and denote either $I=\M^+_{\mathcal{A}}:=\sup_{a\in \mathcal{A}} L_a $ or $I=\M^-_{\mathcal{A}}:= \inf_{a\in \mathcal{A}} L_a$.  Then $\M^-_{ \mathcal{A}} (u_1-u_2) + V(u_1-u_2)\le 0$ and $\M^-_{\mathcal{A}} (u_2-u_1) + V(u_2-u_1)\le 0$ in $G$ in the viscosity sense (by Lemma \ref{lemma:eqdifference} and Remark~\ref{rmk-lemma2}). 
Since $\M^+_{ \mathcal{A}} (-u)=-\M^-_{ \mathcal{A}}u$ and so
 $\lambda_1^+(\M^+_{ \mathcal{A}}+V,G)=\lambda_1^-(\M^-_{\mathcal{A}}+V,G)$, we have anyways $\lambda_1^-(\M^-_{\mathcal{A}}+V,G) >0$, from which the minimum principle holds for $\M^-_{\mathcal{A}}+V$ to ensure that $u_1 =u_2$ in $\R^N$. 

Take $\{G_j\}_{j\in \N}$ an increasing sequence of $C^2$ bounded domains such that $G_j \nearrow \R^N$, and define, for every $j\in \N$, $u_j$ as the unique viscosity solution to 
\[\begin{cases}
\I u_j + Vu_j = 0 & \text{in } G_j \\
u_j = 1 & \text{in } \R^N \setminus G_j
\end{cases}\]
given by the preceding argument to every $G_j$. 
By Lemma \ref{lemma:smp}, $u_j>0$ in $G_j$ and by the maximum principle we get $u_j\leq 1$. Moreover, fix a $x_0 \in G_1$ and redefine $u_j$ by posing $\Tilde{u}_j = \dfrac{u_j}{u_j(x_0)} $.

Let $K\subset \R^N$ be a compact set with $K\subset \subset G_{j}$ for $j \geq j_0$. 
Since $\tilde u_j (x_0) =1$, we may use the Harnack inequality (Proposition \ref{harnack}) in $K$ to get $\tilde u_j \le C_K$, for any $j \ge j_0$, so $\|\Tilde{u}_j\|_{C(K)}\leq C_K$. By Proposition \ref{prop:regularity} we get that $\|\Tilde{u}_j\|_{C^{\alpha}(K)} \leq \Tilde{C}$, so $(\Tilde{u}_j)$ converges strongly in $C(K)$ to a function $\psi\in C(K)$. This means that $(\Tilde{u}_j)_{j\in\N}$ converges to $\psi$ in $C_{\text{loc}}(\R^N)$.

Now, we have to check $\Tilde{u}_j \rightarrow u$ in $L^1_{\omega_s}(\R^N)$. Note that we must have $ 0 \leq \psi \leq C_K$ by the maximum and minimum principles. As constant functions are in $\Lomega$, we conclude $\Tilde{u}_j \rightarrow \psi$ in $\Lomega$ by the dominated convergence theorem.

Finally, by Lemma \ref{lemma:stability}, $\psi$ is a viscosity solution to $\I\psi + V\psi = 0$ in $\R^N$. Since $\psi(x_0) =1$, again by  Lemma \ref{lemma:smp} we must have $\psi >0$ everywhere.
\end{proof}

We now turn to the proof of our main theorem.

\begin{proof}[Proof of Theorem \ref{theo:landis}]
First, let us assume $I$ is of the form $Iu(x) = \sup_{a\in \A} L_au(x)$. In this case, for any $u,v$
\[I(u-v) \geq Iu-Iv.\]
Denote also $I^*u = \inf_{a\in \A} L_a u$.

Let $\psi>0$ a positive viscosity solution given by Proposition \ref{prop:positivesol}, which we normalize to have $\|\psi\|_{\Lomega} = 1$. Note that, for any $R\geq 1$, by Theorem \ref{theo:harnack} we have
\begin{equation} 
\label{eq:decay_psi}
1 = \|\psi\|_{\Lomega} \leq C R^{N+2s}\inf_{B_R}\psi .
\end{equation}

Let $\delta >0 $ be arbitrary. By hypothesis, there exists $R\geq 1$ such that
\begin{equation}
\label{eq:decay_u}
|u(x)|\leq \delta |x|^{-(N+2s)},\textrm{ for all } |x|\geq R,
\end{equation}
so, for any $x\in \R^N \setminus B_R$, we can replace $|x| = R$ in \eqref{eq:decay_psi} and combine it with \eqref{eq:decay_u} to obtain
\begin{align*}
|u(x)|&\leq C\delta \inf_{B_{|x|}} \psi \\
&\leq C\delta \psi(x),
\end{align*} 
so we have $|u|\leq C\delta\psi$ outside $B_R$.

Defining $w_+(x) = u - C\delta \psi$, we have $w_+ \in C(\overline{B_R})\cap \Lomega$, $w_+ \leq 0$ outside $B_R$ and
\begin{align*}
Iw_+ + Vw_+ &\geq Iu + Vu -C\delta (I\psi  + V\psi) = 0 ~ \text{ in }B_R
\end{align*}
in the viscosity sense. We then apply the maximum principle to get $w_+\leq 0$ in $\R^N$, which means that $ u \leq C\delta \psi$ in $\R^N$. Taking $\delta \searrow 0$ yields $u \leq 0$. 

Next, we observe that $-u$ satisfies $I^*(-u) +V(-u) = 0$ in the viscosity sense. Moreover, $-u \geq 0$ in $\R^N$, therefore we can apply Theorem \ref{theo:harnack} to obtain
\begin{align*}
\|u\|_{\Lomega} \leq CR^{N+2s} \inf_{B_R} (-u) .
\end{align*}
If $u\not \equiv 0$, then $\|u\|_{\Lomega} \neq 0$, so actually 
\begin{align*}
 C^{-1}\|u\|_{\Lomega}R^{-(N+2s)} \leq \inf_{B_R}(-u) = \inf_{B_R} |u|
\end{align*}
which is a contradiction with hypothesis \eqref{H}. Thus $u\equiv 0$  in $\R^N$. 

Assume now that $I$ is of the form $Iu = \inf_{a\in \A} L_a u$. It holds that
\[I(u+v) \geq Iu + Iv.\]

Let $\varphi$ be the solution obtained by applying Proposition \ref{prop:positivesol} to the operator $\tilde{F}[u] := -F[-u]$, where $F=\I +V$, normalized such that $\|\varphi\|_{\Lomega} = 1$. Whence we apply Theorem \ref{theo:harnack} and get the decay estimate \eqref{eq:decay_psi} for $\varphi$. 

Proceeding as in the sup case, we derive, for every $\delta >0$, that $|u| \leq C\delta \psi$ in $\R^N \setminus B_R$ for $R>0$ sufficiently large. Defining $w_+ = u-C\delta\varphi$, it holds
 \begin{align*}
Iw_+ + Vw_+ &\geq Iu  + I(-C\delta \varphi) + Vu - C\delta V\varphi \\
&= -C\delta(-I[-\varphi] + V\varphi) = 0
\end{align*}
in $B_R$, in the viscosity sense. By the maximum principle, $w_+ \leq 0$ in $\R^N$, and so $u\leq C\delta\varphi$ in $\R^N$. Taking $\delta \searrow 0$ we obtain $u\leq 0$, and by replicating the former argument we conclude $u\equiv 0$ in $\R^N$. 
\end{proof}

\begin{remark}
The above proof remains valid if we consider nonlocal operators of order $s>\frac12$ with a drift term, namely
\[Iu = \sup_{a\in \A}\{ L_a u + b_a \cdot Du\}.\]
The requirement $s>\frac12$ is necessary for the regularity theory used in the proof of Proposition \ref{prop:positivesol}, see \cite{Chang}.

More precisely, the proof Proposition \ref{prop:weakharnack} remains the same, assuming that $|b_a| \leq 1$ for all $a\in \A$. The rescaling then works with $r_0 = (1 + \sup_{a\in \A} \|b_a\|_{L^{\infty}} + \|V\|_{L^{\infty}})^{-\frac1{2s}}$. The rest of the argument holds verbatim. 
\end{remark}

\bigskip

{\small
	
	{\bf Acknowlegments}. We are indebted with Prof.\ Alexander Quaas for interesting discussions on the topic, and for bringing references \cite{BoVa, Felmer_Quaas_Tan_2012} to our attention.

	We would like to express our sincere gratitude to the referee for their generous time and careful attention, which have greatly contributed to improving the manuscript.	

	S.\ Flores was supported by ANID Magíster Nacional grant 22241020.
	
	G.\ Nornberg was supported by Centro de Modelamiento Matemático (CMM) BASAL fund FB210005 for center of excellence from ANID-Chile; and by ANID Fondecyt grant 1220776.
	
	{\bf Data Availability}: Data sharing is not applicable to this article because we do not analyse or generate any datasets.
	
	{\bf Conflict of Interests}: The authors have no competing interests to declare that are relevant to the content of this article.
}

\section{Appendix}

In this section we give a simple proof of the Harnack inequality for Pucci operators $\M_*^+u :=\sup_{L\in \mathcal{L}*} Lu$ and $\M_*^-u :=\inf_{L\in \mathcal{L}*} Lu$ defined on the class $\mathcal{L}*$ of linear operators with homogeneous kernels of degree $-N-2s$  in the form
$
Lu(x) = \int_{\mathbb{R}^N} ( u(x+y)+u(x-y) -2u(x)  )\frac{y^T A y}{|y|^{N+2s+2}} \mathrm{d} y,
$
with $A$ a symmetric $N\times N$ matrix such that $\lambda I \le A\le \Lambda I$.

\begin{lemma}
	\label{prop:ABP}
	Let $p\ge N$, $f\in C(\overline{B_1})$ and $u\in \Lomega\cap \usc (\overline{B_1})$ be a solution to
$\M_*^+u \geq -f$ in $B_1.$
	Then there exists some constant $C>0$ such that 
	\[\sup_{B_{1/2}} u \leq C(\|u\|_{\Lomega} + \|f\|_{L^p(B_1)}).\]
\end{lemma}
\begin{proof}
	Let $\eta \in C(\overline{B_1})$ be the unique viscosity solution to 
	\[\begin{cases}
		\M_*^+\eta = f & \text{ in } B_1 \\
		\eta = 0 & \text{ in }\R^N \setminus B_1.
	\end{cases}\]
	Consider $w = u- \eta$, then
$\M_*^+ w \geq \M_*^+u - \M_*^+ \eta = \M_*^+u - f \geq 0$
	in $B_1$, in the viscosity sense. Then, by \cite[Theorem 3.3.5]{fernandez-real_integro-differential_2024} applied to $w$, we get
	\[ \sup_{B_{1/2}} w \leq C \|w\|_{\Lomega}.\]
	On the other hand, 
	\begin{align*}
		\sup_{B_{1/2}} u &\leq  \sup_{B_{1/2}} w + \sup_{B_{1/2}} \eta 
		\leq C(\|u\|_{\Lomega} + \|\eta\|_{\Lomega}) + \sup_{B_1} \eta \\
		&\leq C\|u\|_{\Lomega} + C\sup_{B_1} |\eta| \int_{B_1} \frac{\dx}{1+|x|^{N+2s}}  +\sup_{B_1} \eta.
	\end{align*}
Now the ABP estimate \cite[Theorem 3.1]{kitano_wsigma_2022} applied to $\eta$ and $-\eta$ (since $\M_*^-(-\eta) = -\M^+_*\eta = -f$) yields 
	\[\sup_{B_1} |\eta |\leq {C} \|f\|_{L^p(B_1)},\]
	obtaining the desired estimate.
\end{proof}
\begin{proposition}[Local Maximum Principle]
	\label{cor:localmp}
	Let $s\in (0,1)$, $C_0\geq 0$, and $V\in C(\overline{B_1})$. If $u\in C(\overline{B_1}) \cap \Lomega$ satisfies
$
		\M_*^+u + Vu \geq -C_0~\text{ in } B_1
$
	in the viscosity sense, then there exists $C = C(N, s, \lambda, \Lambda, \|V\|_\infty)$ such that
	\begin{equation*}
		\label{eq:est_localmp}
		\underset{B_{1/2}}{\sup} ~u \leq C(\|u\|_{\Lomega} + C_0).
	\end{equation*}
\end{proposition}
\begin{proof}
	We apply Lemma \ref{prop:ABP} with $f = C_0 + cu$ and some $p\ge N$ with $p>1$, to get 
	\[\sup_{B_{1/2}} u \leq C(\|u\|_{\Lomega} + \|Vu + C_0\|_{L^p(B_1)}).\]
	By the Hölder inequality and the fact that $|V| \leq M$ we obtain 
	\begin{align*}
		\|Vu + C_0\|_{L^p(B_1)} \leq \|Vu\|_{L^p(B_1)} + |B_1|C_0 \leq M |B_1|^{p/(p-1)} \|u\|_{L^1(B_1)} + |B_1|C_0.
	\end{align*}
	Observing that $\|u\|_{L^1(B_1)}\leq C\|u\|_{L^1_{\omega_s}(B_1)}\leq C\|u\|_{\Lomega}$ for some constant depending on $N$ and $s$, we conclude the theorem.
\end{proof}
Finally, Proposition \ref{harnack}, with $\M_*$  instead of $\M$, is obtained by combining Propositions \ref{prop:weakharnack}  and \ref{cor:localmp}.

\begin{remark}
	By defining $v$ and $\tilde{V}$ as in \eqref{eq:v}, we may extend the Local Maximum Principle and the Harnack inequality to balls of arbitrary radius, as follows:
	Let $R\geq 1$, $s\in (0,1)$ and $V\in C(\overline{B_{R+1}})$, then any viscosity solution $u\in C(\overline{B_{R+1}}) \cap \Lomega$ to
	$	\M_*^+u + Vu \geq -C_0 \; \text{ in }B_{R+1}
	$
is such that
	\begin{align*}
		\label{eq:sharplocalmp}
		\underset{B_R}{\sup}~u \leq C (1 + R^{N+2s})\left(\|u\|_{\Lomega} + C_0 \right).
	\end{align*}
	If, in addition, $u$ satisfies
	$\M_*^-u + Vu \leq C_0  \text{ in }B_{R+1}$ with $u\geq 0 \text{ in } \R^N$,
	we have
	\begin{align*}
		\underset{B_R}{\sup}~u \leq C (1 + R^{2(N+2s)})\left(\inf_{B_R}~ u + C_0 \right).
	\end{align*}
	Here,  the constant $C>0$ depends on $N$, $s$, $\lambda$, $\Lambda$ and $\|V\|_{L^{\infty}(B_{R+1})}$.
\end{remark}

\bibliography{Landis.bib}
\bibliographystyle{abbrv}

\end{document}